\documentclass[11pt]{amsart}

\usepackage[colorlinks=true,linkcolor=blue,citecolor=blue]{hyperref}
\usepackage{amsfonts, amssymb, amsmath}
\usepackage{mdframed, graphicx, color}

\usepackage{bbm, a4wide}

\usepackage{caption,subfig}

\DeclareMathAlphabet{\mymathbb}{U}{bbold}{m}{n}

\newcommand{\RR}{\mathbb{R}}
\newcommand{\ZZ}{\ts\mathbb{Z}}
\newcommand{\XX}{\mathbb{X}}

\newcommand{\QQ}{\mathbb{Q}}

\newcommand{\cT}{\mathcal{T}}

\newcommand{\ts}{\hspace{0.5pt}}

\newtheorem{theorem}{Theorem}

\newtheorem{coro}[theorem]{Corollary}
\newtheorem{lemma}[theorem]{Lemma}
\theoremstyle{definition}

\newtheorem{example}[theorem]{Example}
\newtheorem{remark}[theorem]{Remark}

\newcommand{\exend}{\hfill$\Diamond$}

\title[Dynamical Invariants from Asymptotic Composants]
  {Dynamical Invariants from Asymptotic Composants}

\author{Franz G\"{a}hler}
\address{Fakult\"at f\"ur Mathematik, Universit\"at Bielefeld, \newline
  \indent  Postfach 100131, 33501 Bielefeld, Germany}
\email{gaehler@math.uni-bielefeld.de}

\begin{document}

\begin{abstract}

  Asymptotic composants and their incidence relations are powerful
  invariants of 1-dimensional inflation tilings spaces, which can
  distinguish many MLD classes of tilings. In particular, and unlike most
  other invariants, they can often provide obstructions to a tiling space
  being MLD to its reflection. We present a simple algorithm to
  determine these asymptotic composants for primitive inflation tiling
  spaces in one dimension, and illustrate how they can be used to tell
  different MLD classes of tilings apart. In an Appendix, we then show that
  the structure of asymptotic composants, together with the orbit separation
  dimension (OSD), can distinguish \emph{all} MLD classes of inflations
  tilings with pure-point spectrum for a bunch of small inflation factors,
  which illustrates the power of these invariants.
  
\end{abstract}

\keywords{Asymptotic composants, Inflation tilings, Mutual local derivability}
\subjclass{37B52, 52C23} 

\maketitle

\section{Introduction}\label{sec:intro}

Let $\XX$ be a space of aperiodic tilings of the real line $\RR$,
generated by a primitive inflation rule $\varrho$. It is well known
that the translation action on such a space is not equicontinuous.
In fact, it has been shown \cite[Corollary~3.9]{BD} that there exist
pairs of tilings $(\cT_1,\cT_2)$ which are \emph{right asymptotic},
in the sense that $d(\cT_1 - t, \cT_2 - t) \rightarrow 0$ as
$t \rightarrow \infty$, where $d$ is the usual tiling metric.
Right asymptotic tilings completely agree on an entire right half-line.
Up to translations, there exists at least one such pair, and at most
finitely many. Analogously, there also exist \emph{left asymptotic}
pairs of tilings, which agree on an entire left half-line. We call
any tiling \emph{asymptotic}, if it is left or right asymptotic to
some other tiling. It has also been shown \cite[Corollary~3.9]{BD}
that each translation orbit of an asymptotic pair contains one pair
that is fixed by some power of the inflation $\varrho$. The (finite)
set of asymptotic inflation fixed points, together with the two
equivalence relations of being left resp.\ right asymptotic is an
interesting dynamical invariant, as it is one of the rare invariants
which can possibly distinguish a space from its reflected space
(with inverted dynamics).

In the context of a 1D tiling space $\XX$, the notion
of a composant (arc component) is identical to that of a path
component \cite{BD}, which in turn can be identified with a
translation orbit. Accordingly, we call the composants of asymtotic
tilings asymptotic composants. A pair of left asymptotic composants
approach each other towards left infinity, and analogously, a pair of
right asymptotic composants approach each other at right infinity.
There is an obvious bijection between asymptotic composants and the
asymptotic inflation fixed points representing them. We adopt here the
term asymptotic composants.

We concentrate here on the 1D case, although Barge and Olimb \cite{BO}
have introduced the \emph{branch locus} as a corresponding concept for
higher dimensions. While their approach brings geometry into the
discussion, and hence opens the door towards studying the interplay
between geometric symmetries and MLD relations, the branch locus is
also a much more complicated object, and its systematic use would go
far beyond the scope of the present paper.

The remainder of this note is organised as follows. In
Section~\ref{sec:invar}, we discuss invariants derived from the set of
asymptotic composants and their equivalence relations of being left
resp. right asymptotic. These invariants must be preserved under
mutual local derivability (MLD) relations, see \cite{TAO} for background,
which is a special kind of topological conjugacy of the corresponding
translation dynamical systems. These invariants thus impose constraints
on the existence of MLD equivalence relations between two tiling spaces.

In Section~\ref{sec:deriv}, we then describe how the asymptotic
composants and their equivalence relations can be determined in
a fully algorithmic way. This is a substantial improvement over
\cite{BD}, where a rather involved procedure is given,
which is not easily automatised. In Section~\ref{sec:examp},
we then illustrate the application and the power of these
invariants with a number of examples. Further examples are
discussed in an Appendix, where we show that the MLD classes of
all primitive, irreducible inflations with an inflation factor
that is one of the eight smallest ternary Pisot units can be
distinguished by comparing their asymptotic composants and their
orbit separation dimension (OSD), as developed in \cite{BGG}.

\section{Invariants from asymptotic composants}\label{sec:invar}

Suppose we have two spaces $\XX$ and $\tilde{\XX}$ of 1D aperiodic
tilings, generated by primitive inflation rules $\varrho$ and
$\tilde{\varrho}$ sharing a common inflation factor $\lambda$. Our
question now is: can these two spaces possibly be MLD to each other?
We approach the answer to this question by analysing and comparing the
structure of the asymptotic composants of the two spaces. Recall that
there is a bijection between the set of asymptotic composants of $\XX$
and the subset $S\subset\XX$ consisting of all asymptotic tilings
which are fixed by some power of the inflation.  We call these tilings
briefly \emph{asymptotic inflation fixed points}.  The set $S$ is
finite and non-empty \cite[Corollary~3.9]{BD}. Also, $\tilde{\XX}$
has a corresponding set $\tilde{S}$ of asymptotic inflation fixed points.
Any MLD map $\XX\rightarrow\tilde{\XX}$ must induce an isomorphism between
$S$ and $\tilde{S}$, and it must also preserve various equivalence relations
on $S$ and $\tilde{S}$. In order to work this out, the following lemmas
will be useful.

\begin{lemma}\label{lemma:equiv}
  The images under an MLD map of two tilings which are left (right)
  asymptotic are again left (right) asymptotic. Likewise, the
  inflations of two tilings which are left (right) asymptotic are
  again left (right) asymptotic.
\end{lemma}
\begin{proof}
  Two tilings are left (right) asymptotic if and only if they agree
  on a complete left (right) half-line. This property is naturally
  preserved under MLD maps and inflations.
\end{proof}

Recall that an inflation map with inflation factor $\lambda$ on a
tiling space $\XX$  is a self-homeo\-morphism $\varrho'$ of $\XX$
such that a scaled tiling $\lambda\cT$ is MLD to $\varrho'(\cT)$.
As a consequence, all inflation maps on $\XX$ with the same inflation
factor are MLD equivalent. The question now is, whether there are
non-trivial MLD self-homeomorphisms which preserve the scaling
center of an inflation map.

\begin{lemma}\label{lemma:infl}
  Let $\XX$ be a space of 1D aperiodic tilings, generated by a primitive
  inflation rule $\varrho$ with scaling factor $\lambda$ and scaling centre
  $c$. If $\XX$ has pure-point dynamical spectrum, there can be no second
  inflation map on $\XX$ with the same scaling factor and the same scaling centre.
\end{lemma}
\begin{proof}
  For a second inflation rule $\varrho'$ with the same scaling factor and the
  same scaling centre, we would need a non-trivial MLD map which preserves
  all fibres over the maximal equicontinuous factor (MEF) of the underlying
  dynamical system (for a discussion of the MEF in our context, compare
  \cite{BK13}). In case the dynamical system hast pure-point spectrum,
  almost all fibres have cardinality $1$, and in the remaining fibres,
  all tilings from the same fibre differ only on a set of zero density.
  In particular, there are regions of unbounded size, in which two tilings
  from the same fibre agree completely. This implies that there can be no
  non-trivial MLD map between such tilings.
\end{proof}

\begin{remark}\label{rem:MLDsym}
  The pure-point condition above is essential. For instance, the Thue--Morse
  tiling space (compare \cite[Example~7]{BRY}) can be generated by two different
  inflation maps. In this case, almost all fibres over the MEF have cardinality 2,
  and the two tilings in a typical fibre are related to each other by a swap of
  the two tile types, which is an MLD map. Consequently, if we combine an
  inflation with a swap of the two tile types, we obtain a second possible
  inflation map, so that the inflation map is no longer unique. In the
  pure-point case, two different inflation maps can only be related by a
  translation, which is an MLD map, but it does not preserve the scaling centre.
  \exend
\end{remark}

In the following, we concentrate on inflation tilings with pure-point
spectrum, which have a unique inflation for a given scaling factor and
scaling centre. In the general case, our conclusion will be weaker,
compare Remark~\ref{rem:nonpp}.

Suppose now that we have two spaces $\XX$ and $\tilde{\XX}$ of 1D
aperiodic tilings with pure-point spectrum, generated by primitive
inflations rules $\varrho$ and $\tilde{\varrho}$ with the same
inflation factor $\lambda$. Assume there is an MLD map
$f:\XX \rightarrow \tilde{\XX}$, which in view of Lemma~\ref{lemma:equiv}
must map the scaling centre of $\varrho$ to that of $\tilde{\varrho}$, 
and by Lemma~\ref{lemma:infl} must satisfy
$\tilde{\varrho} = f \circ \varrho \circ f^{-1}$.

Denote by $S$ the set of asymptotic inflation fixed points of $\XX$,
and by $\tilde{S}$ the corresponding set for $\tilde{\XX}$. By
Lemma~\ref{lemma:equiv}, $f$ must induce an isomorphism beween $S$ and
$\tilde{S}$. Moreover, $S$ carries two partitions, $P_L$ and $P_R$.
Two tilings from $S$ are in the same subset of $P_L$ ($P_R$) if and
only if they are left (right) asymptotic to each other. $\tilde{S}$
carries two analogous partitions, $\tilde{P}_L$ and $\tilde{P}_R$,
and $f$ must also induce an isomorphism between the partitions $P_L$
and $\tilde{P}_L$, and between $P_R$ and $\tilde{P}_R$. Further, by
Lemma~\ref{lemma:equiv} the inflation $\varrho$ acts via a
permutation $p$ on the set $S$, and $\tilde{\varrho}$ acts via
a permutation $\tilde{p}$ on $\tilde{S}$. Due to our choice of $f$,
the map $f$ preserves the orbit struction of these permutations, and
induces a bijection between the orbits of $p$ and the orbits of
$\tilde{p}$. Hence, we can partition the set $S$ into subsets $S_j$,
each containing the tilings in a single orbit of the permutation $p$.
Likewise, there is a corresponding partition of $\tilde{S}$ into
subsets $\tilde{S}_j$, containing the orbits of $\tilde{p}$, and we
can order these subsets such that $f$ induces an isomorphism between
$S_j$ and $\tilde{S}_j$ for each $j$. Since $f$ preserves the orbit
structure of the inflation, we can now constrain the partitions
$P_L$ and $P_R$ to partitions $P^j_L$ and $P^j_R$ of a fixed subset
$S_j$, and correpondingly form partitions $\tilde{P}^j_L$ and
$\tilde{P}^j_R$ of $\tilde{S}_j$. Then, $f$ must induce isomorphisms
between $P^j_L$ and $\tilde{P}^j_L$ for every $j$, and likewise between
$P^j_R$ and $\tilde{P}^j_R$ for every $j$. We summarise these findings
as follows.

\begin{theorem}\label{thm:equiv}
  Let $\XX$ and $\tilde{\XX}$ be two spaces of 1D aperiodic tilings
  with pure-point spectrum, generated by primitive inflations $\varrho$
  and $\tilde{\varrho}$ with the same inflation factor $\lambda$.
  Suppose that there exists an MLD map $f:\XX \rightarrow \tilde{\XX}$
  as above, which then must induce isomorphisms
  \begin{itemize}
    \item between the permutations $p$ of $S$ and $\tilde{p}$ of $\tilde{S}$,
    \item between $S$ and $\tilde{S}$, and between $S_j$ and
      $\tilde{S}_j$ for every $j$, 
    \item between $P_L$ and $\tilde{P}_L$, and between $P^j_L$ and
      $\tilde{P}^j_L$ for every $j$,
    \item between $P_R$ and $\tilde{P}_R$, and between $P^j_R$ and
      $\tilde{P}^j_R$ for every $j$.
  \end{itemize}
  If any of these isomorphisms fails, the two spaces cannot be
  MLD to each other. \hfil\qed
\end{theorem}

\begin{remark}\label{rem:power}
  Theorem~\ref{thm:equiv} compares inflations with the same inflation
  factor. However, since any integer power $\varrho^n$ of an inflation
  $\varrho$ defines the same tiling space as $\varrho$, by passing to
  suitable powers we can sometimes make inflations with different
  inflations factors $\lambda,\lambda'$ comparable. Note, though,
  that two inflations can be MLD only if their inflation factors
  satisfy $\QQ(\lambda)=\QQ(\lambda')$ \cite[Corollary~2.3]{BJV}.
  \exend
\end{remark}

An interesting application of Theorem~\ref{thm:equiv} is
to the special case when, for a fixed $\alpha<0$, we have
$\tilde{\XX} = \{ \alpha\cT: \cT \in \XX \}$.  In other words,
$\tilde{\XX}$ is the mirror image of $\XX$, possibly with an
additional rescaling (which does not affect the structure of the
asymptotic composants). If $\XX$ is generated by an inflation
$\varrho$, $\tilde{\XX}$ is then generated by a inflation
$\tilde{\varrho}$, which sends tiles to the same patches as $\varrho$,
but in the reversed order.  Our question is again: can the two spaces
$\XX$ and $\tilde{\XX}$ possibly be MLD to each other? In the context
of constant-length inflation subshifts, this question was addressed
already in \cite[Theorem~1,Fact~3]{BRY}. Clearly, rescaling by
$\alpha<0$ is a homeomorphism, but it cannot be an MLD homeomorphism,
because it reverses the dynamics and the space orientation.
So, we must look for a different map.

By construction, $S$ and $\tilde{S}$ are isomorphic, and so are
$p$ and $\tilde{p}$. More interesting is the observation that
$P_L$ is isomorphic to $\tilde{P}_R$, and $P^j_L$ is isomorphic to
$\tilde{P}^j_R$ for every $j$, and likewise with $L$ and $R$ exchanged.
Hence, the following consequence is immediate.

\begin{coro}\label{coro:equiv}
  Let $\XX$ be a space of 1D aperiodic tilings with pure-point spectrum,
  generated by a primitive inflation $\varrho$, and suppose there exists
  an MLD map $f:\XX \rightarrow \tilde{\XX} = \{ \alpha\cT: \cT \in \XX \}$
  for some fixed $\alpha<0$. With the data introduced above, $f$ must
  then induce isomorphisms between $P_L$ and $P_R$, and between $P^j_L$
  and $P^j_R$, for every $j$. If any of these isomorphism fails, the two
  spaces cannot be MLD to each other. \hfil\qed
\end{coro}

\begin{remark}\label{rem:nonpp}
  In Theorem~\ref{thm:equiv} and Corollary~\ref{coro:equiv}, we have
  assumed that the dynamical system has pure-point spectrum, which ensures
  that the inflation is unique. Without this assumption, the claims
  made must be weakened. If there is not a unique inflation, we cannot
  conclude that the permutation actions on the asymptotic composants must
  be equivalent, nor that the partitions per inflation orbit must be
  isomorphic. However, for two systems $(\XX,\RR)$ and $(\tilde{\XX},\RR)$
  to be MLD, the partitions $P_L$ and $\tilde{P}_L$ still must
  be isomorphic, and similarly for $P_R$ and $\tilde{P}_R$. Likewise,
  for an MLD class of tilings to be mirror-symmetric, the partitions
  $P_L$ and $P_R$ must be isomorphic. These weaker constraints may still
  be useful for distinguishing two systems. The same constraints also hold,
  if we have only a general topological conjugacy of two dynamical system,
  not necessarily an MLD one. The set of asymptotic composants with their
  incidence relations is in fact a topological property.
  \rightline{\hfil\exend}
\end{remark}

\section{Determining  the asymptoptic composants}\label{sec:deriv}

Barge and Diamond \cite{BD} have given a procedure to determine the
set of asymptotic inflation fixed points and their equivalence relations,
but it is rather involved and complicated. Often, the inflation needs
to be rewritten in terms of new tiles, and sometimes even several
rewritings have to be considered in parallel, which makes it difficult
to automatize the process. Here, we propose a much simpler procedure,
which is easy to automatise and implement on a computer. The key to
avoid the problems encountered in \cite{BD} is to use large enough
seeds for the asymptotic inflation fixed points.

We start by determining the pairs of left asymptotic inflation fixed
points. The corresponding right asymptotic ones are then determined in
an analogous way. Our aim is to determine pairs of seeds that are large
enough to accommodate both the scaling centre of the inflation and the
left splitting point of the pair. For a pair of patches with the same
left-most vertex, the left splitting point is the right-most common vertex
with the property, that to the left of it the two patches agree completely.
To construct such pairs of patches, we start with the collection $C_k$
of all legal patches of $k$ consecutive tiles, for $k$ large enough.
We do not know a priori how large that is, but since there are only
finitely many asymptotic inflation fixed points, some finite $k$ will
be enough, and if $k$ is too small, we will notice that and can start
afresh with a larger $k$. From the collection $C_k$ we then form a
collection $W_k$ of pairs $(p,q)$ of $k$-patches from $C_k$, with the
property that $p$ and $q$ agree on the first, left-most tile, but
disagree on the second tile. The left splitting point of these pairs
of patches therefore is their second vertex from the left.

We now inflate each pair $(p,q) \in W_k$ and determine the left splitting
point of the inflated pair $(\varrho(p),\varrho(q))$. If to the right
of this splitting point, any of the patches $\varrho(p)$ and $\varrho(q)$
has fewer than $k$ tiles, then $k$ was too small, and we start
afresh with a larger $k$. Note that requiring $k$ tiles to the right
of the left splitting point (rather than $k-1$ tiles) ensures that
the patches are growing in both directions under inflation; compare
Figure~\ref{fig:seeds}. We now assume that all pairs in $W_k$ pass this
test. The inflation then induces a map $f:W_k \rightarrow W_k$. A pair
$(p,q)$ is mapped to the pair of $k$-subpatches of $\varrho(p)$ resp.\
$\varrho(q)$, which start one tile to the left, and end $k-1$ tiles
to the right of the left splitting point of the inflated pair
$(\varrho(p),\varrho(q))$. This map is not surjective in general, but
we can iterate it until we reach a stable image $W'_k \subset W_k$, on
which $f$ acts by permutation. The pairs in $W'_k$ are the seeds for the
left asymptotic pairs of inflation fixed points we are after.

What remains to be done is to determine the position of the splitting
point of the pair relative to the scaling centre of the inflation (which
we put at the origin). Suppose we place a pair $(p,q)$ with its left splitting
point at $s \le 0$ (with $s>0$, iterated inflation would generate a pair of
identical tilings). The inflation scaling maps this splitting point
to $\lambda s$, but the true left splitting of $(\varrho(p),\varrho(q))$
will be at $\lambda s + \delta(p,q)$, where $\delta(p,q) \ge 0$ is
determined from the pair $(\varrho(p),\varrho(q))$. In order that the
splitting point remains stable under inflation, we need
$s = \lambda s + \delta(p,q)$, or $s = \delta(p,q)/(1-\lambda) \le 0$.
With this, we now know all seeds of left asymtotic inflation fixed points
with absolute position in space. In principle, any smaller seed containing
the origin in its interior would be enough to determine the inflation fixed
points, but the larger seeds containing the splitting point as well guarantee
that we get only the \emph{asymptotic} inflation fixed points, and allow to
determine which of the tilings are actually left asymptotic to each other.

\begin{figure}[t]
\centerline{\includegraphics[width=6cm]{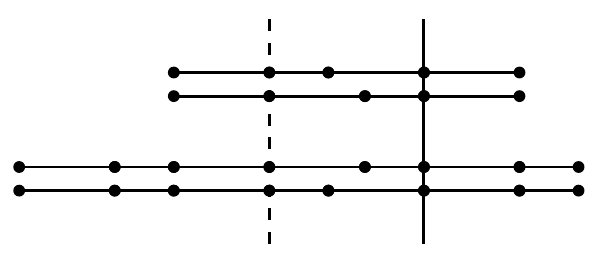}}
\caption{For the Fibonacci inflation $\varrho = [ab,a]$, seeds for the (only)
  left asymptotic pair of inflation fixed points are shown at the top, with
  their inflations underneath. If we place the scaling centre at the position
  of the solid vertical line, the position of the left splitting point
  (dashed vertical line) remains stable under inflation. Note that we
  need patches of at least four tiles; patches of three tiles would not
  grow beyond the scaling centre, and a unique continuation would not be
  guaranteed. Note also that the inflation swaps the two members of the pair.
  \label{fig:seeds}}
\end{figure}

Figure~\ref{fig:seeds} illustrates the situation up to this point.
We can now split the pairs of seeds again, to obtain the
collection of all left asymptotic inflation fixed points, $S_L$. From
$S_L$, we have to remove duplicates (a tiling can occur in several
asymptotic pairs), but we keep the information of which tilings are
left asymptotic to each other in the form of a partition $P'_L$ of $S_L$.
In a completely analogous way, the corresponding collection $S_R$ of all
right asymptotic inflation fixed points can be determined, along
with a partition $P'_R$ of $S_R$. The total set of all asymptotic
inflation fixed points $S$ is the union of $S_L$ and $S_R$,
where two seeds from $S_L$ and $S_R$ have to be identified if they
agree on an overlap containing the origin in its interior. The partition
$P'_L$ can then be completed to the partition $P_L$ of $S$,
by adding singletons for all tilings which occur only in $S_R$, and
analogously $P'_R$ can be completed to $P_R$. With the collected data,
all ingredients necessary for the application of Theorem~\ref{thm:equiv}
or Corollary~\ref{coro:equiv} are now ready to be used.

\section{Examples}\label{sec:examp}

In this section, we illustrate the application of Theorem~\ref{thm:equiv}
and Corollary~\ref{coro:equiv} with a detailed discussion of a number of
simple examples. A more thorough analysis of MLD relations between large
sets of inflation tilings is then given in the Appendix. We closely follow
the notation and the arguments of \cite{MLD}. Patches of tiles (or entire
tilings) are represented by words in a finite alphabet. Each letter stands
for a tile, and a word for a patch of tiles arranged in a given order.
A symbol $a$ in such a word can be interpreted in at least two ways.
On the one hand, it represents a geometric tile, here an interval of
length $\ell_a$, where the different lengths are chosen proportional
to the components of the left Perron--Frobenius (PF) eigenvector of
the inflation matrix, so that under inflation, each patch is mapped
to a new patch which is exactly $\lambda$ times longer, where $\lambda$
is the PF eigenvalue of the inflation matrix (compare \cite[Chapter~4]{TAO}).

On the other hand, each tile type can be regarded as a generator of a
free group, and each patch as a word in these generators. The advantage
of this view is that the inflation map can be regarded as a free group
endomorphism: each generator is replaced by a word in the generators.
We denote such an inflation map $\varrho$ by the list of the images of
the generators: $\varrho = [\varrho(a),\varrho(b),\ldots]$. Words
representing tilings are special, in that they contain no inverses
of generators. MLD maps $\theta: \XX \rightarrow \tilde{\XX}$
between two tiling spaces have to respect this. They can then be
interpreted as free group automorphisms, which transform words
from $\XX$ to words from $\tilde{\XX}$ \cite{MLD}. This underlying
group structure makes it much easier to \emph{find} MLD maps. This
is particularly so for unimodular, binary or ternary inflations
with irreducible characteristic polynomial of the inflation matrix.
For these, it has been shown \cite{TWZ04} that the required MLD
group automorphisms are products of particularly simple,
\emph{Fibonacci-type} transformations, combined with inner automorphisms
and permutations of tile types.

\begin{example}\label{ex:one}
As a first example, we consider the space $\XX$ generated by the primitive
inflation $\varrho = [\varrho(a),\varrho(b)] = [aab,ba]$, and compare it
to its mirror image $\tilde{\XX}$, generated by the reversed inflation
$\tilde{\varrho} = [baa,ab]$. It is easy to see that the two spaces
are different, and hence asymmetric. However, this does not yet exclude
the existence of a (non-symmetry-preserving) MLD relation. It is also
easy to see that there are four asymptotic inflation fixed point tilings
in total, with seeds $a.a$, $a.b$, $b.a$, and $b.b$. They form two
left asymptotic and two right asymptotic pairs, but each left asymptotic
pair is not right asymptotic, and vice versa. So far, this looks pretty
symmetric. The action of the inflation breaks this symmetry, however.
Indeed, $\varrho$ permutes two tilings in $\XX$ forming a right asymptotic
pair, which however is not left asymptotic. The same is true for the other
right asymptotic pair. Hence, for both $j$, the partitions $P^j_L$ and $P^j_R$
cannot be isomorphic, and an MLD relation is impossible.
\exend
\end{example}

\begin{example}\label{ex:two}
As a second example, we consider the Tribonacci space $\XX$, generated
by the inflation $\varrho = [c,ca,cb]$. It is well known that this space
is symmetric. In fact, the reversed inflation, $\tilde{\varrho} = [c,ac,bc]$
generates exactly the same space (the two free group endomorphisms are
conjugate by an inner automorphism). Of course, the asymptotic composants
are compatible with this. There are 6 asymptotic inflation fixed points,
forming one left asymptotic triple, and one right asymptotic triple,
which are mapped onto each other by reflection. The inflation permutes
the members of each triple cyclically.

However, there are other, non-symmetric tilings spaces in the same MLD
class. These come in mirror pairs. As detailed in \cite[Example 2]{MLD},
conjugating the inflation $\varrho$ with the automorphism $\theta_L = [a,ba,c]$
yields the inflation $\varrho^{}_L = \theta_L \circ \varrho \circ \theta^{-1}_L
=[c,a,cab]$. This automorphism induces an MLD transformation: It splits the
$b$ tile into an $a$ tile and a rest, which becomes the new $b$ tile.
The inflation matrix of $\varrho^{}_L$ differs from that of $\varrho$,
but the two are conjugate in $\textrm{GL}_3(\ZZ)$. The reversed automorphism
$\theta_R = [a,ba,c]$ can be used to conjugate $\tilde{\varrho}$,
which yields $\varrho^{}_R = \theta_R \circ \tilde{\varrho} \circ
\theta^{-1}_R =[c,a,bac]$. Thus, $\varrho^{}_L$ and $\varrho^{}_R$ form a mirror
pair of inflations, which produce a mirror pair of tilings spaces, both of
which are MLD with the symmetric Tribonacci tiling space, and thus are MLD
to each other.
\exend
\end{example}

\begin{example}\label{ex:three}
As a third example, we consider the mirror pair of inflations
$\varrho^{}_L = [bc,a,b]$ and $\varrho^{}_R= [cb,a,b]$, whose inflation
factor $\lambda$ is the smallest PV number, also known as the plastic 
number \cite[Example~4.4]{TAO}. The spaces $\XX_L$ and $\XX_R$ generated by
these inflations are not reflection symmetric. For instance, $\XX_L$ has three
left asymptotic pairs and two right asymptotic triples of asymptotic inflation
fixed points.  More interesting are the products
$\varrho = \varrho_L\circ\varrho_R = [ba,bc,a]$ and
$\tilde{\varrho} = \varrho_R\circ\varrho_L = [ab,cb,a]$, whose asymptotic
composants are symmetric, with symmetric inflation action. Indeed, since
$\varrho_L$ and $\varrho_R$ are invertible, the two inflations are conjugate
to each other: $\tilde{\varrho} = \varrho_L^{-1}\circ\varrho\circ\varrho_L$, and
$\varrho = \varrho_R^{-1}\circ\tilde{\varrho}\circ\varrho_R$. There is one
subtlety, however: this conjugation, which is an inflation with scaling $\lambda$,
changes the relative scale. $\XX$ and its reflection $\tilde{\XX}$ are MLD,
but only if one of them is first rescaled by an \emph{odd} power of $\lambda$.
This is an instance where in Corollary~\ref{coro:equiv} a factor $\alpha\ne-1$
has to be chosen. This example is completely analogous to \cite[Example~3]{MLD},
where inflations with factor $\lambda^3$ are found to be conjugate, provided
they are compared at the right relative scale.
\exend
\end{example}

\begin{figure}
\centerline{\includegraphics[width=12cm]{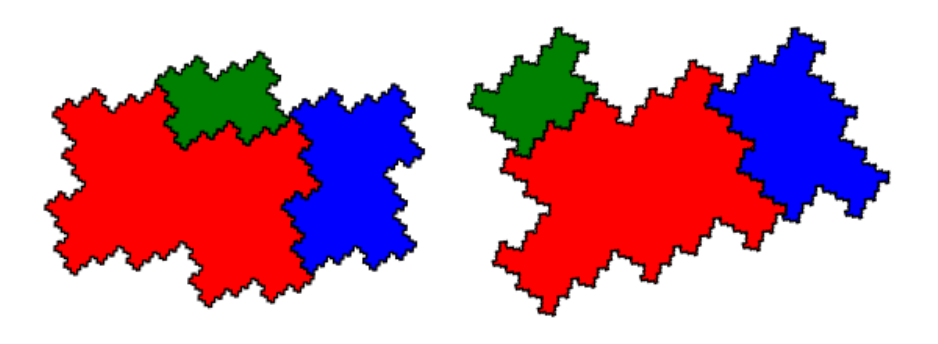}}
\caption{Windows of the tiling spaces with inflation $\varrho_M = [aca,a,b]$
  (left) and $\varrho_R = [aac,a,b]$ (right).\label{fig:kol}}
\end{figure}

\begin{example}\label{ex:four}
As a fourth example, we consider the inflations $\varrho^{}_M = [aca,a,b]$,
$\varrho^{}_L = [caa,a,b]$, and $\varrho1{}_R = [aac,a,b]$, all with the same
inflation matrix. $\varrho^{}_M$ is in the (manifestly reflection symmetric)
MLD class of the Kolakoski-(3,1) sequence \cite{Kol}. In Figure~\ref{fig:kol}
we show the windows of the tiling spaces generated by $\varrho^{}_M$ and $\varrho^{}_R$.
Their window boundaries have a very similar fractal structure, and in fact have
the same Hausdorff dimension, as can be computed via the orbit separation dimension
\cite{BGG}, which makes it tempting to suspect that they are in the same MLD
class. Yet, the inflations $\varrho^{}_L$ and $\varrho^{}_R$ generate spaces with
asymmetric asymptotic composants, which are inequivalent to those of the
Kolakoski-(3,1) space. The space of $\varrho^{}_M$ has two left asymtotic and
two right asymptotic pairs, whereas that of $\varrho^{}_R$ has two left asymptotic
pairs and one right asymptotic triple. Hence, the three inflations $\varrho^{}_M$,
$\varrho^{}_L$, and $\varrho¬{}_R$ produce tilings which belong to three distinct
MLD classes of tiling spaces, and in fact to three distict topological conjugacy
classes (comapare also Remark~\ref{rem:nonpp}).
\exend
\end{example}

\section{Appendix: MLD classification by asymptotic composants and OSD}

In this appendix, we give a more extensive analysis of the MLD relations
between large classes of 1D inflation tilings. It turns out that,
in conjunction with the orbit separation dimension (OSD) \cite{BGG},
the structure of the asymptotic composants is a powerful tool for
distinguishing MLD classes of such tilings. We illustrate this by
classifying all primitive inflation tilings having as inflation factor
one of the eight smallest ternary Pisot units. These inflations all have
pure-point spectrum, as has been derived earlier \cite{AGL14}.
But before that, as a warm-up, we do the same for all inflations having
as inflation matrix one of the first three powers of the Fibonacci
inflation matrix.

As a first step, for a given collection of inflations with the same
inflation factor, we try to find explicit MLD transformations between
them, using the techniques explained in \cite{MLD}; see also
Section~\ref{sec:examp}. For the binary and ternary examples
considered in this Appendix, the results of \cite{TWZ04} are
particularly helpful for finding MLD relations. Another technique
is to check whether two inflations have equivalent return word
encodings \cite{Durand}, which also implies an MLD relation.
Once we have partitioned the inflations into collections which are
known to be MLD, we then try to make sure that inflations from different
collections are not MLD, by computing various MLD invariants which can
differentiate between the different classes. For all examples presented
below, this process was successful, and it turned out that OSD and
asymptotic composants are sufficient to distinguish all MLD classes
in this collection.

In the tables below, we give, for each MLD class, a representative
inflation rule in the form of the list of images of the prototiles
under the inflation. For the Fibonacci inflation, this would be
$[ab,a]$. The second entry is the OSD \cite{BGG}, an approximate real
number $\ge 1$. The third entry is the (\v{C}ech) cohomology rank
\cite{AP}, which is not really needed, but still is useful information.
Then comes the structure of the asymptotic composants. For these, the
asymptotic inflation fixed points are numbered, and under the header
\emph{AC left} we group the numbers of those inflation fixed points
which are left asymptotic to each other, under \emph{AC right} the
numbers of those which are right asymptotic to each other, and under
the header \emph{Perm} we give the permutation by which the inflation
acts on the set of asymptotic inflation fixed points.

Below we only list the results, without giving detailed computations
(most of which have been obtained by computer), but we complement
those results with some further remarks.

\subsection{Variations of Fibonnaci and its powers}

In Table~\ref{tab:fibo}, we collect MLD class representatives of all
primitive binary inflations with an inflation matrix which is one of
the first three powers of the Fibonacci inflation matrix. These are
obtained by varying the order of the tiles within the supertiles.
We obtain two symmetric and seven mirror pairs of MLD classes. The
symmetric ones (numbers 1 and 5) are distinguished by their OSD and
their cohomology rank. Except for these two classes, all the others
have asymmetric asymptotic composants, or an asymmetric inflation
action on them. Numbers 2 and 5 happen to have the same OSD,
but are distinguished by their cohomology rank and their asymptotic
composants. Number 7 has been discussed in Example~\ref{ex:one}.

\begin{table}[t!] 
  \caption{Variations of powers of the Fibonacci inflation, obtained
    by reshuffling the tiles within the inflation words. Each MLD class
    of such inflations is listed with its properties and a representative
    inflation.}
\begin{center}
{\small
\begin{tabular}{|clcrlll|}
  \hline
  \textbf{Nr} & \textbf{Repr} & \textbf{Rk} & \textbf{OSD} & \textbf{AC left} & \textbf{AC right} & \textbf{Perm} \\
  \hline
  1 & [ab,a]      & 2 &  1.0000  & [1,2]       & [1,2]       & (1,2)      \\
  2 & [aaabb,aab] & 2 &  4.1842  & [1,3]       & [2,4]       & (2,4)      \\
  3 & [abbaa,aab] & 2 & 10.4419  & [3,4]       & [1,2]       & (3,4)      \\
  4 & [abaab,baa] & 3 &  2.5676  & [1,2],[3,4] & [1,3],[2,4] & (1,3)(2,4) \\
  5 & [abbba,aba] & 3 &  4.1842  & [1,2],[3,4] & [1,3],[2,4] & (1,4)(2,3) \\
  6 & [aabab,baa] & 3 &  4.8878  & [1,2],[3,4] & [1,3],[2,4] & (1,3)(2,4) \\
  7 & [aab,ba]    & 3 & 11.8744  & [1,2],[3,4] & [1,3],[2,4] & (1,3)(2,4) \\
  8 & [aaabb,baa] & 3 & 22.6225  & [1,2],[3,4] & [1,3],[2,4] & (1,3)(2,4) \\
  \hline
\end{tabular}
}
\end{center}
\label{tab:fibo}
\end{table}

\subsection{Inflations by powers of the plastic number}

In Table~\ref{tab:plast}, we collect MLD class representatives of all
inflations with inflation factor the plastic number $\lambda_1$, or
its second or third power (compare \cite[Example~4.4]{TAO}).
The plastic number, which is the smallest of all Pisot numbers,
is the largest root of the polynomial $x^3-x-1$.
We find one symmetric and seven mirror pairs of MLD classes. The
representative of class 7, which together with its reflection is the
only one with inflation factor $\lambda_1$, generates a mirror pair of
MLD classes with rather large OSD. If we take the composition of this
inflation with its reflection, we obtain a representative of the only
symmetric MLD class (number 1). This inflation was discussed in detail
in Example~\ref{ex:two}. All other MLD classes have asymmetric
asymptotic composants, or an asymmetric inflation action on them.
Apart from mirror images, all MLD classes have distinct OSD.
We note that the dragons of \cite[Figs.~5-8]{MLD} belong to class 2.

\begin{table}[t!]
  \caption{MLD class representatives of inflations by powers of the plastic number.}
\begin{center}
{\small
\begin{tabular}{|clcrlll|}
  \hline
  \textbf{Nr} & \textbf{Repr} & \textbf{Rk} & \textbf{OSD} & \textbf{AC left} & \textbf{AC right} & \textbf{Perm} \\
  \hline
  1 & [ab,cb,a]    & 3 &  2.2229 & [2,8],[3,4]             & [1,7],[5,6]             & (1,5,7,6)(2,3,8,4)    \\
  2 & [aac,a,bc]   & 3 &  4.4260 & [1,4]                   & [2,5],[3,6]             & (2,6,5,3)             \\
  3 & [acb,ab,bc]  & 3 &  5.9468 & [3,4],[5,6]             & [1,2]                   & (3,6,4,5)             \\ 
  4 & [bac,ab,cb]  & 3 & 21.4521 & [1,2,3]                 & [4,7],[5,6]             & (1,2)(4,5,7,6)        \\ 
  5 & [acb,ba,bca] & 4 & 14.6784 & [2,3],[4,5]             & [1,3,5],[2,4]           & (2,4)(3,5)            \\
  6 & [caa,a,bc]   & 5 & 17.5023 & [1,2,3],[4,5,6]         & [1,4],[2,5],[3,6]       & (1,3,2)(4,6,5)        \\ 
  7 & [bc,a,b]     & 5 & 37.3375 & [1,2],[3,4],[5,6]       & [1,3,5],[2,4,6]         & (1,3,5)(2,4,6)        \\
  8 & [cba,ab,bc]  & 7 & 19.8019 & [1,2,3],[4,5,6],[7,8,9] & [1,4,7],[2,5,8],[3,6,9] & (1,3,2)(4,6,5)(7,9,8) \\
  \hline
\end{tabular}
}
\end{center}
\label{tab:plast}
\end{table}

\subsection{Inflations with Tribonacci inflation factor}

In Table~\ref{tab:tribo}, we collect MLD class representatives of all ternary, primitive
inflations with the Tribonacci inflation factor, $\lambda_2 \approx 1.83929$, the largest
root of the polynomial $x^3-x^2-x-1$. Up to simultanous permutations of rows and
columns, there are $4$ inflation matrices with this characteristic polynomial.
The corresponding inflations fall into one symmetric and four mirror pairs of
MLD classes. The Tribonacci inflation belongs to the symmetric MLD class
(number 1 in Table~\ref{tab:tribo}), as do two further mirror pairs of inflations
(compare Example~\ref{ex:three}. All other MLD classes have asymmetric asymptotic
composants, or an asymmetric inflation action on them, and are thus asymmetric.
Except for mirror pairs, all MLD classes have distinct OSD.

\begin{table}[t!]
  \caption{MLD class representatives of inflations by the Tribonacci inflation factor.}
\begin{center}
{\small
\begin{tabular}{|clcrlll|}
  \hline
  \textbf{Nr} & \textbf{Repr} & \textbf{Rk} & \textbf{OSD} & \textbf{AC left} & \textbf{AC right} & \textbf{Perm} \\
  \hline
  1 & [ab,ac,a]  & 3 &  2.2060 & [4,5,6]                 & [1,2,3]                  & (1,2,3)(4,6,5)       \\
  2 & [ab,ca,a]  & 3 &  9.6109 & [3,6],[4,5]             & [1,2]                    & (1,2)(3,5,6,4)       \\
  3 & [bac,a,b]  & 5 &  5.5585 & [1,2],[3,4],[5,6]       & [1,3,5],[2,4,6]          & (1,4,5,2,3,6)        \\    
  4 & [bbc,cb,a] & 7 & 21.2133 & [1,2,3],[4,5,6],[7,8,9] & [1,4,7],[2,5,8],[3,6,9]  & (1,5,3,4,2,6)(7,8,9) \\
  5 & [acc,ca,b] & 7 & 27.9209 & [1,2,3],[4,5,6],[7,8,9] &  [1,4,7],[2,5,8],[3,6,9] & (1,4,7)(2,6,8,3,5,9) \\   
  \hline
\end{tabular}
}
\end{center}
\label{tab:tribo}
\end{table}

\subsection{Inflations by powers of the second-smallest ternary Pisot unit}

In Table~\ref{tab:second}, we collect MLD class representatives of all inflations
with inflation factor the second-smallest ternary Pisot unit $\lambda_2$, or its
square, where $\lambda_2$ is the largest root of the polynomial $x^3-x^2-1$.
Except for the mirror pair with number 1 in Table~\ref{tab:second}, all
inflations have inflation factor $\lambda_2^2$. We have one inflation matrix with
leading eigenvalue $\lambda_2$, and seven with leading eigenvalue $\lambda_2^2$.
We find one symmetric and 17 mirror pairs of MLD classes. If we take the
composition of the representative of class 1 with its reflection,
we obtain a representative of the only symmetric MLD class (number 7).
This is an analogous phenomenon as discussed in Example~\ref{ex:three}.
The mirror pair of class 1 and the symmetric MLD class 7 have the same OSD.
All other mirror pairs have distinct OSD. All mirror pairs have asymmetric
asymptotic composants, or an asymmetric inflation action in them,
so that they are really asymmetric. We note that the examples of
\cite[Figs.~1,2]{MLD} belong to the (square of) class 1 in the table.

\begin{table}[t!]
  \caption{MLD class representatives of inflations by powers of the second-smallest ternary Pisot unit.}
\begin{center}
{\small
\begin{tabular}{|clcrlll|}
  \hline
  \textbf{Nr} & \textbf{Repr} & \textbf{Rk} & \textbf{OSD} & \textbf{AC left} & \textbf{AC right} & \textbf{Perm} \\
  \hline
  1 & [ac,a,b]     & 3 &   3.7829 & [4,7],[5,6]             & [1,2,3]                 & (1,2,3)(4,5,7,6)     \\
  2 & [bcc,ac,cb]  & 3 &   8.7737 & [1,2],[3,4]             & [2,4]                   & (1,3)(2,4)           \\
  3 & [ccbbc,a,cb] & 3 &  12.8251 & [4,5],[6,7]             & [1,2,3]                 & (1,3,2)(4,6,5,7)     \\
  4 & [bcc,ca,cb]  & 3 &  15.5091 & [4,5],[6,7]             & [1,2,3]                 & (1,2,3)(4,6)(5,7)    \\
  5 & [abc,ca,a]   & 3 &  27.6619 & [3,4],[5,6]             & [1,2]                   & (1,2)(3,6,4,5)       \\
  6 & [cccbb,a,cb] & 3 &  65.4734 & [3,4],[5,6]             & [1,2]                   & (1,2)(3,5,4,6)       \\
  7 & [bac,ac,a]   & 4 &   3.7829 & [1,2],[3,4]             & [1,3],[2,4]             & (1,4)(2,3)           \\
  8 & [cab,ac,a]   & 5 &  12.2329 & [1,2],[3,4],[5,6]       & [1,3,5],[2,4,6]         & (1,4,5,2,3,6)        \\
  9 & [accb,ca,b]  & 5 &  21.6026 & [1,2,3],[4,5,6]         & [1,4],[2,5],[3,6]       & (1,4)(2,6)(3,5)      \\
 10 & [bbac,a,b]   & 5 &  26.5686 & [1,2],[3,4],[5,6]       & [1,3,5],[2,4,6]         & (1,4,5,2,3,6)        \\
 11 & [bacc,ca,b]  & 5 &  36.0646 & [1,2],[3,4],[5,6]       & [1,3,5],[2,4,6]         & (1,4,5,2,3,6)        \\
 12 & [cabc,ac,b]  & 6 &   7.6330 & [1,2,3],[4,5,6]         & [1,4],[2,5],[3,6],[7,8] & (1,6,2,4,3,5)(7,8)   \\
 13 & [bccbc,a,cb] & 7 &  24.8793 & [1,2,3],[4,5,6],[7,8,9] & [1,4,7],[2,5,8],[3,6,9] & (1,8,4,2,7,5)(3,9,6) \\
 14 & [acbc,ca,b]  & 7 &  27.4558 & [1,2,3],[4,5,6],[7,8,9] & [1,4,7],[2,5,8],[3,6,9] & (1,4,7)(2,6,8,3,5,9) \\
 15 & [bcbcc,a,cb] & 7 &  36.3827 & [1,2,3],[4,5,6],[7,8,9] & [1,4,7],[2,5,8],[3,6,9] & (1,8,4,2,7,5)(3,9,6) \\
 16 & [accc,cca,b] & 7 &  71.8547 & [1,2,3],[4,5,6],[7,8,9] & [1,4,7],[2,5,8],[3,6,9] & (1,4,7)(2,6,8,3,5,9) \\
 17 & [abcc,ca,b]  & 7 &  97.2913 & [1,2,3],[4,5,6],[7,8,9] & [1,4,7],[2,5,8],[3,6,9] & (1,4,7)(2,6,8,3,5,9) \\
 18 & [bbccc,a,cb] & 7 & 174.3150 & [1,2,3],[4,5,6],[7,8,9] & [1,4,7],[2,5,8],[3,6,9] & (1,8,4,2,7,5)(3,9,6) \\
  \hline
\end{tabular}
}
\end{center}
\label{tab:second}
\end{table}

\begin{table}[t!]
  \caption{MLD class representatives of inflations by the of Kolakoski-$(3,1)$ inflation factor.}
\begin{center}
{\small
\begin{tabular}{|clcrlll|}
  \hline
  \textbf{Nr} & \textbf{Repr} & \textbf{Rk} & \textbf{OSD} & \textbf{AC left} & \textbf{AC right} & \textbf{Perm} \\
  \hline
  1 & [aca,a,b]   & 3 &  2.5534 & [2,8],[3,4]             & [1,7],[5,6]              & (1,5,7,6)(2,3,8,4)   \\
  2 & [aac,a,b]   & 3 &  2.5534 & [4,7],[5,6]             & [1,2,3]                  & (1,2,3)(4,5,7,6)     \\
  3 & [cabc,a,cb] & 3 &  3.3299 & [4,5],[6,7]             & [1,2,3]                  & (1,3,2)(4,6,5,7)     \\  
  4 & [abc,cbc,a] & 3 &  7.2657 & [3,6],[4,5]             & [1,2]                    & (1,2)(3,4,6,5)       \\
  5 & [abc,ccb,a] & 3 &  8.7747 & [4,7],[5,6]             & [1,2,3]                  & (1,2)(4,6,7,5)       \\   
  6 & [bcc,ab,cb] & 3 &  9.6121 & [1,2,3]                 & [4,7],[5,6]              & (1,2)(4,5,7,6)       \\   
  7 & [bac,bcb,a] & 3 &  9.4088 & [3,4],[5,6]             & [1,2]                    & (1,2)(3,6,4,5)       \\   
  8 & [acb,ba,b]  & 4 & 27.2585 & [1,2],[3,4]             & [1,3],[2,4]              & (1,3)(2,4)           \\ 
  9 & [bac,cba,a] & 5 & 13.8073 & [1,2,3],[4,5,6]         & [1,4],[2,5],[3,6],[7,8]  & (1,5,3,4,2,6)        \\   
 10 & [abc,ba,b]  & 5 & 17.9987 & [1,2],[3,4],[5,6]       & [1,3,5],[2,4,6]          & (1,3,5)(2,4,6)       \\
 11 & [acb,bcc,a] & 5 & 18.8489 & [1,2],[3,4],[5,6]       & [1,3,5],[2,4,6]          & (1,3,5)(2,4,6)       \\   
 12 & [bcc,ba,cb] & 5 & 19.9170 & [1,2],[3,4],[5,6]       & [1,3,5],[2,4,6]          & (1,5,3)(2,6,4)       \\   
 13 & [bcac,a,cb] & 7 & 11.4483 & [1,2,3],[4,5,6],[7,8,9] & [1,4,7],[2,5,8],[3,6,9]  & (1,8,4,2,7,5)(3,9,6) \\
 14 & [bac,ccb,a] & 7 & 44.8306 & [1,2,3],[4,5,6],[7,8,9] & [1,4,7],[2,5,8],[3,6,9]  & (1,5,3,4,2,6)(7,8,9) \\  
 15 & [bacc,a,cb] & 7 & 60.9140 & [1,2,3],[4,5,6],[7,8,9] & [1,4,7],[2,5,8],[3,6,9]  & (1,8,4,2,7,5)(3,9,6) \\
  \hline
\end{tabular}
}
\end{center}
\label{tab:kola}
\end{table}

\subsection{Inflations by the factor of Kolakoski-$(3,1)$}

In Table~\ref{tab:kola}, we collect MLD class representatives of all ternary, primitive
inflations with the Kolakoski-$(3,1)$ inflation factor $\lambda_4 \approx 2.20557$, the
largest root of the polynomial $x^3-2x^2-1$ \cite{Kol}. Up to simultanous permutations of
rows and columns, there are $6$ inflation matrices with this characteristic polynomial.
The corresponding inflations fall into one symmetric and 14 mirror pairs of
MLD classes. The Kolakoski-$(3,1)$ inflation belongs to the symmetric MLD class
(number 1 in Table~\ref{tab:kola}), as do three further mirror pairs of inflations.
All other MLD classes have asymmetric asymptotic composants, or an asymmetric
inflation action on them, and are thus asymmetric. Except for the symmetric
MLD class and one mirror pair, all other mirror pairs have distinct OSD.

\subsection{Three real eigenvalues}

In Table~\ref{tab:threereal}, we collect MLD class representatives of all ternary, primitive
inflations, whose inflation factor is the smallest ternary Pisot unit with two real Galois
conjugates. This inflation factor $\lambda_5 \approx 2.24698$ is the largest
root of the polynomial $x^3-x^2-x-1$. Up to simultanous permutations of rows
and columns, there are $5$ inflation matrices with this characteristic polynomial.
The corresponding inflations fall into 12 mirror pairs of MLD classes.
All mirror pairs have have distinct OSD, and all have asymmetric asymptotic
composants, or an asymmetric inflation action on them, so that all spaces
are really asymmetric.

\begin{table}[t]
  \caption{MLD class representatives of inflations by the smallest ternary Pisot unit with only real Galois conjugates.}
\begin{center}
{\small
\begin{tabular}{|clcrlll|}
  \hline
  \textbf{Nr} & \textbf{Repr} & \textbf{Rk} & \textbf{OSD} & \textbf{AC left} & \textbf{AC right} & \textbf{Perm} \\
  \hline
   1 & [aac,ac,b]  & 3 &  3.6679 & [1,4],[2,5]             & [1,2,3]                 & (1,2)(4,5)           \\       
   2 & [bac,ca,bc] & 3 &  9.8655 & [1,2,4]                 & [3,7],[5,6]             & (1,2)(3,5,7,6)       \\
   3 & [bac,ca,cb] & 3 & 11.5683 & [4,5],[6,7]             & [1,2,3]                 & (1,2,3)(4,6)(5,7)    \\
   4 & [bbac,a,bc] & 3 & 14.2701 & [1,2]                   & [3,5],[4,6]             & (1,2)(3,4,5,6)       \\  
   5 & [abc,ba,a]  & 4 &  5.3822 & [1,2],[3,4]             & [1,3],[2,4]             & (1,3)(2,4)           \\     
   6 & [abc,ac,cb] & 4 & 12.9699 & [1,2],[3,4]             & [1,3],[2,4],[5,6]       & (1,3)(2,4)(5,6)      \\  
   7 & [acb,ca,bc] & 5 &  5.1217 & [1,2,3],[4,5],[6,7]     & [2,4,6],[3,5,7]         & (2,3)(4,7)(5,6)      \\ 
   8 & [bbc,ba,cb] & 5 & 10.2679 & [1,2],[3,4],[5,6],[7,8] & [1,3,5],[2,4,6]         & (1,5,3)(2,6,4)(7,8)  \\
   9 & [abc,ca,cb] & 5 & 12.3481 & [1,2],[3,4],[5,6]       & [1,3,5],[2,4,6]         & (1,3,5)(2,4,6)       \\
  10 & [aac,ca,b]  & 7 &  8.7780 & [1,2,3],[4,5,6],[7,8,9] & [1,4,7],[2,5,8],[3,6,9] & (1,8,4,2,7,5)(3,9,6) \\
  11 & [bbac,a,cb] & 7 & 10.2377 & [1,2,3],[4,5,6],[7,8,9] & [1,4,7],[2,5,8],[3,6,9] & (1,4,7)(2,6,8,3,5,9) \\
  12 & [babc,a,cb] & 7 & 50.1704 & [1,2,3],[4,5,6],[7,8,9] & [1,4,7],[2,5,8],[3,6,9] & (1,8,4,2,7,5)(3,9,6) \\
  \hline
\end{tabular}
}
\end{center}
\label{tab:threereal}
\end{table}

\section*{Acknowledgements}

The author would like to thank M.~Baake, J.~Maz\'{a}\v{c}, and N.~Ma\~{n}ibo
for fruitful discussions.
This work was supported by the German Research Council (Deutsche
Forschungsgemeinschaft, DFG) under SFB-TRR 358/1 (2023) -- 491392403.

\end{document}